\documentclass[a4paper,11pt]{article}
\usepackage{enumerate}
\usepackage{amsmath,amsthm,amssymb}
\usepackage{tikz,hyperref}
\usepackage[mathlines]{lineno}

\parskip \bigskipamount

\title{Improved Bounds for the Graham-Pollak Problem for Hypergraphs}
\author{Imre Leader\thanks{Department of Pure Mathematics and Mathematical Statistics, Centre for Mathematical Sciences, University of Cambridge, Wilberforce Road, Cambridge CB3 0WB, United Kingdom. Email: \texttt{I.Leader@dpmms.cam.ac.uk}.} \and Ta Sheng Tan\thanks{Institute of Mathematical Sciences, Faculty of Science, University of Malaya, 50603 Kuala Lumpur, Malaysia. Email: \texttt{tstan@um.edu.my}.}}

\newtheorem{thm}{Theorem}

\newtheorem{corollary}[thm]{Corollary}

\theoremstyle{remark}  

\theoremstyle{definition}

\begin{document}

\maketitle

\begin{abstract}
 For a fixed $r$, let $f_r(n)$ denote the minimum number of complete $r$-partite $r$-graphs needed to partition the complete $r$-graph on $n$ vertices. 
 The Graham-Pollak theorem asserts that $f_2(n)=n-1$. 
 An easy construction shows that $f_r(n) \leq (1+o(1))\binom{n}{\lfloor r/2 \rfloor}$, and we write $c_r$ for the least number such that $f_r(n) \leq c_r (1+o(1))\binom{n}{\lfloor r/2 \rfloor}$.

 It was known that $c_r < 1$ for each even $r \geq 4$, but this was not known for any odd value of $r$. 
 In this short note, we prove that $c_{295}<1$. Our method also shows that $c_r \rightarrow 0$, answering another open problem.
\end{abstract}

\emph{Keywords: }Hypergraph, Decomposition, Graham-Pollak

\section{Introduction}

The edge set of $K_n$, the complete graph on $n$ vertices, can be partitioned into $n-1$ complete bipartite subgraphs: this may be done in many ways, for example by taking $n-1$ stars centred at different vertices. Graham and Pollak~\cite{graham1,graham2} proved that the number $n-1$ cannot be decreased. Several other proofs of this result have been found, by Tverberg~\cite{tverberg}, Peck~\cite{peck}, and Vishwanathan~\cite{vishwanathan1,vishwanathan2}, among others.

Generalising this to hypergraphs, for $n\ge r\ge 1$, let $f_r(n)$ be the minimum number of complete $r$-partite $r$-graphs needed to partition the edge set of $K_n^{(r)}$, the complete $r$-uniform hypergraph on $n$ vertices
(i.e., the collection of all $r$-sets from an $n$-set).
Thus the Graham-Pollak theorem asserts that $f_2(n) = n-1$. For $r\ge 3$, an easy upper bound of $\binom{n-\lceil r/2\rceil}{\lfloor r/2\rfloor}$ may be obtained by generalising the star example above. 
Indeed, for $r$ even, having ordered the vertices, consider the collection of $r$-sets whose $2nd, 4th, \ldots, rth$ vertices are fixed. This forms a complete $r$-partite $r$-graph, and the collection of all $\binom{n- r/2}{ r/2 }$ such is a partition of $K_n^{(r)}$. 
For $r$ odd, we instead fix the $2nd, 4th, \ldots, (r-1)th$ vertices, yielding a partition into $\binom{n- (r+1)/2}{ (r-1)/2 }$ parts.

Alon~\cite{alon1} showed that $f_3(n) = n-2$. More generally, for each fixed $r\ge 1$, he showed that  
\begin{equation*}
 \frac{2}{\binom{2\lfloor r/2 \rfloor}{\lfloor r/2 \rfloor}}(1+o(1))\binom{n}{\lfloor r/2\rfloor}\le f_r(n)\le (1-o(1))\binom{n}{\lfloor r/2\rfloor},
\end{equation*}
where the upper bound follows from the construction above. 
Writing $c_r$ for the least $c$ such that $f_r(n) \leq c (1+o(1))\binom{n}{\lfloor r/2 \rfloor}$, the above results assert that $c_2 = 1$, $c_3 = 1$, and $\frac{2}{\binom{2\lfloor r/2 \rfloor}{\lfloor r/2 \rfloor}} \le c_r \le 1$ for all $r$. 
How do the $c_r$ behave?

Cioab\v{a}, K\"{u}ndgen and Verstra\"{e}te~\cite{cioaba1} gave an improvement (in a lower-order term) to Alon's lower bound, and Cioab\v{a} and Tait~\cite{cioaba2} showed that the construction above is not sharp in general, but Alon's asymptotic bounds (i.e., the above bounds on $c_r$) remained unchanged. 
Recently, Leader, Mili\'{c}evi\'{c} and Tan~\cite{leader} showed that $c_r \leq \frac{14}{15}$ for each even $r \geq 4$.
However, they could not improve the bound of $c_r\le1$ for any odd $r$ -- the point being that the construction above is better for $r$ odd than for $r$ even (the exponent of $n$ is $(r-1)/2$ for $r$ odd versus $r/2$ for $r$ even), and so is harder to improve.

In this note, we give a simple argument to show that $c_{295}<1$. Our method also shows that $c_r \rightarrow 0$, answering another question from \cite{leader}. 

It would be interesting to know what happens for smaller odd values of $r$: for example, is $c_5<1$?
Determining the precise value of $c_4$ (i.e., the asymptotic behaviour of $f_4(n)$) would also be of great interest, as would determining the decay rate of the $c_r$. See ~\cite{leader} for several related questions and conjectures. 

\section{Main Result}
The motivation for our proof is as follows.
The key to the approach used in $\cite{leader}$ in proving $c_r<1$ for each even $r\ge 4$ was to investigate the minimum number of products of complete bipartite graphs, that is, sets of the form $E(K_{a,b})\times E(K_{c,d})$, needed to partition the set $E(K_n)\times E(K_n)$.
Writing $g(n)$ for this minimum value, it is trivial that $g(n)\le (n-1)^2$, by taking the products of the complete bipartite graphs appearing in a decomposition of $K_n$ into $n-1$ complete bipartite graphs.
It was shown in \cite{leader} that $g(n)\le \left(\frac{14}{15}+o(1)\right)n^2$.
It turned out that this upper bound on $g(n)$ was enough (via an iterative construction) to bound $c_r$ below 1 for each even $r\ge 4$. 

Now, as remarked above, for $r$ odd the construction in the Introduction is much better than for $r$ even. In fact, while there are many iterative ways to redo the construction when $r$ is even, passing from $n/2$ to $n$, these fail when $r$ is odd: it turns out that an extra factor is introduced at each stage.
However, rather unexpectedly, we will see that (at least if $r$ is large) if we partition into \emph{many} pieces, instead of just two pieces, then the gain we obtain from the $14/15$ improvement in $g(n)$ outweighs the loss arising from this extra factor -- even though this extra factor grows as the number of pieces grows.

A \emph{minimal decomposition} of a complete $r$-partite $r$-graph $K_n^{(r)}$ is a partition of the edge set into $f_r(n)$ complete $r$-partite $r$-graphs. 
A \emph{block} is a product of the edge sets of two complete bipartite graphs.
Similarly, a \emph{minimal decomposition} of $E(K_n)\times E(K_n)$ is a partition of $E(K_n)\times E(K_n)$ into $g(n)$ blocks.
Finally, for a set $V$, we may write $E(V)$ to denote the edge set of the complete graph on $V$, that is, the set of all 2-subsets of $V$.

\begin{thm}\label{theorem_main}
 Let $r=2d+1$ be fixed. Then for each $k$ there exists $\epsilon_k$, with $\epsilon_k \rightarrow 0$ as $k\rightarrow \infty$, such that for all $n$ we have 
 \begin{equation*}                                                                                                                               
f_r(kn) \le \left(\left(\frac{14}{15}\right)^{\left\lfloor\frac{d}{2}\right\rfloor}+d\left(\frac{14}{15}\right)^{\left\lfloor\frac{d-1}{2}\right\rfloor}+\epsilon_k\right)(1+o(1))\binom{kn}{d}.
 \end{equation*}
 (Here the $o(1)$ term is as $n\rightarrow \infty$, with $k$ and $d$ fixed.)
\end{thm}

\begin{proof}
 In order to decompose the edge set of $K_{kn}^{(r)}$, we start by splitting the $kn$ vertices into $k$ equal parts, say $V\left(K_{kn}^{(r)}\right) = V_1\cup V_2\cup \cdots \cup V_k$, where $|V_i|= n$ for each $i$. 
 We consider the $r$-edges based on their intersection sizes with the $k$ vertex classes. 
 For each partition of $r$ into positive integers $r_1+r_2+\cdots+r_l$ with $r_1\le r_2\le \cdots \le r_l$ and for each collection of $l$ vertex classes $V_{i_1}, V_{i_2}, \ldots, V_{i_l}$, the set of $r$-edges $e$ with $|e\cap V_{i_j}| = r_j$ for all $j$ can be decomposed into $f_{r_1}(n)f_{r_2}(n)\cdots f_{r_l}(n)$ complete $r$-partite $r$-graphs: take a complete $r_j$-partite $r_j$-graph from a minimal decomposition of $K_n^{(r_j)}$ for each $j$, and form a complete $r$-partite $r$-graph by taking the product of them. 
 
 Note that if at least three values of the $r_j$ are odd, then $f_{r_1}(n)f_{r_2}(n)\cdots f_{r_l}(n) = O(n^{d-1})$, as $f_{s}(n)\le \binom{n}{\lfloor s/2\rfloor}$ for any $s$.
 So the set of $r$-edges $e$ with $|e\cap V_i|$ is odd for at least three distinct $V_i$ can be decomposed into $Cn^{d-1}$ complete $r$-partite $r$-graphs, for some constant $C$ depending on $d$ and $k$. 
 
 Let $C'$ be the number of partitions of $r$ into at most $d-1$ positive integers where exactly one of them is odd. Then we observe that the set of $r$-edges $e$ such that $e$ intersects with at most $d-1$ vertex classes and $|e\cap V_i|$ is odd for exactly one $V_i$ can be decomposed into at most $C'k^{d-1}n^{d}$ complete $r$-partite $r$-graphs. 
 
 We are now only left with two partitions of $r$: $r=1+2+2+\cdots+2$ and $r=2+2+\cdots+2+3$. 
 The first case corresponds to the set of $r$-edges with $r_1=1, r_2=\cdots=r_{d+1}=2$.  
 For each of the $\binom{k}{d}$ collections of $d$ vertex classes $V_{i_1},V_{i_2},\ldots,V_{i_d}$, we claim that the set of $r$-edges $\{e:|e\cap V_{i_j}|=2, j=1,2,\ldots,d\}$ can be decomposed into $g(n)^{ d/2 }$ or $ng(n)^{(d-1)/2}$ complete $r$-partite $r$-graphs, depending on whether $d$ is even or odd. 
 This is done by pairing up the $V_{i_j}$s (or all but one of the $V_{i_j}$s if $d$ is odd), and forming complete $r$-partite $r$-graphs using products of blocks in a minimal decomposition of $E(K_n)\times E(K_n)$. 
 [For example, for $d=4$, we would take a decomposition of $E(V_{i_1})\times E(V_{i_2})$ into blocks $E_x \times F_x, 1\le x\le g(n)$, and similarly a decomposition of $E(V_{i_3})\times E(V_{i_4})$ into blocks $G_x \times H_x, 1\le x\le g(n)$, and now the set of all $9$-edges $e$ with $|e\cap V_{i_j}| = 2$ for all $1\le j\le 4$ may be decomposed into $g(n)^2$ complete 9-partite 9-graphs by taking the $E_x\times F_x \times G_y\times H_y \times (V_{i_1}\cup V_{i_2}\cup V_{i_3}\cup V_{i_4})^c$ for $1\le x,y\le g(n)$.]
 
 Finally, the second case corresponds to the set of $r$-edges with $r_1=r_2=\cdots=r_{d-1} = 2, r_{d} = 3$. These can be decomposed in a similar fashion. Indeed, for each collection of $d$ vertex classes $V_{i_1},V_{i_2},\ldots,V_{i_d}$, the set of $r$-edges $\{e:|e\cap V_{i_d}|=3\mbox{ and } |e\cap V_{i_j}|=2, j=1,2,\ldots,d-1\}$ can be decomposed into $n^2g(n)^{(d-2)/2}$ or $ng(n)^{(d-1)/2}$ complete $r$-partite $r$-graphs, depending on whether $d$ is even or odd. There are $d\binom{k}{d}$ such sets of $r$-edges.
 
 Combining the above and the bound on $g(n)$, we have
 \begin{align*}
  f_r(kn) &\le 
  \begin{cases}
   \binom{k}{d}g(n)^{\frac{d}{2}} + d\binom{k}{d}n^2g(n)^{\frac{d-2}{2}}  +C'k^{d-1}n^{d}   +Cn^{d-1} &\quad\mbox{(if $d$ even)}\\
   \binom{k}{d}ng(n)^{\frac{d-1}{2}} + d\binom{k}{d}ng(n)^{\frac{d-1}{2}}  +C'k^{d-1}n^{d}   +Cn^{d-1} &\quad\mbox{(if $d$ odd)}
  \end{cases}\\
  &\le  \binom{k}{d}\left(\frac{14}{15}\right)^{\left\lfloor\frac{d}{2}\right\rfloor}n^d + d\binom{k}{d}\left(\frac{14}{15}\right)^{\left\lfloor\frac{d-1}{2}\right\rfloor}n^d + C'k^{d-1}n^d + o(n^d)\\
  &\le \left(\left(\frac{14}{15}\right)^{\left\lfloor\frac{d}{2}\right\rfloor} + d\left(\frac{14}{15}\right)^{\left\lfloor\frac{d-1}{2}\right\rfloor} + \frac{d!C'}{k}\right)\binom{k}{d}n^d + o(n^d)\\
  & \le \left(\left(\frac{14}{15}\right)^{\left\lfloor\frac{d}{2}\right\rfloor}+d\left(\frac{14}{15}\right)^{\left\lfloor\frac{d-1}{2}\right\rfloor}+\epsilon_k\right)(1+o(1))\binom{kn}{d}.
 \end{align*}
 \end{proof}
 
 \begin{corollary}
  Let $r \ge 295$ be a fixed odd number. Then there exists $c< 1$ such that $$f_r(n) \le c(1+o(1))\binom{n}{\lfloor r/2 \rfloor}.$$
 \end{corollary}
 \begin{proof}
  As above, write $r=2d+1$. It is straightforward to check that for $d\ge 147$ we have $\left(\frac{14}{15}\right)^{\left\lfloor\frac{d}{2}\right\rfloor}+d\left(\frac{14}{15}\right)^{\left\lfloor\frac{d-1}{2}\right\rfloor} <1$. 
  Choosing $k$ such that $$c = \left(\frac{14}{15}\right)^{\left\lfloor\frac{d}{2}\right\rfloor}+d\left(\frac{14}{15}\right)^{\left\lfloor\frac{d-1}{2}\right\rfloor} + \epsilon_k < 1,$$
  we have $f_r(kn) \le c(1+o(1))\binom{kn}{d}$ for all $n$. 
  However since the function $f_r(n)$ is monotone in $n$, and $k$ is constant as $n$ varies, it follows that $f_r(n) \le c(1+o(1))\binom{n}{d}$ for all $n$.
 \end{proof} 

 From Theorem~\ref{theorem_main}, we have $$c_{2d+1} \le \left(\frac{14}{15}\right)^{\left\lfloor\frac{d}{2}\right\rfloor}+d\left(\frac{14}{15}\right)^{\left\lfloor\frac{d-1}{2}\right\rfloor}$$ for every $d$. 
 Also, it is easy to see that $c_{2d}\le c_{2d+1}$. Indeed, by excluding a vertex in the complete $(2d+1)$-graph on $n+1$ vertices, the complete $(2d)$-partite $(2d)$-graphs induced from the complete $(2d+1)$-partite $(2d+1)$-graphs in a minimal decomposition of $K_{n+1}^{(2d+1)}$ form a decomposition of $K_{n}^{(2d)}$, implying that $f_{2d}(n)\le f_{2d+1}(n+1)$. Hence we have the following.
 
 \begin{corollary}\label{cor_cr}
 The numbers $c_r$ satisfy 
  $$ c_r \le \frac{r}{2}\left(\frac{14}{15}\right)^{r/4}+o(1).$$
 \end{corollary}
 
 Corollary~\ref{cor_cr} implies that $c_r\rightarrow 0$ as $r\rightarrow \infty$, proving Conjecture 16 in \cite{leader}.


\begin{thebibliography}{99}
\bibitem{alon1} N. Alon, \emph{Decomposition of the complete $r$-graph into complete $r$-partite $r$-graphs}, Graphs and Combinatorics \textbf{2} (1986) 95--100.

\bibitem{cioaba1} S.M. Cioab\v{a}, A. K\"{u}ndgen and J. Verstra\"{e}te, \emph{On decompositions of complete hypergraphs}, Journal of Combinatorial Theory, Series A \textbf{116} (2009) 1232--1234.

\bibitem{cioaba2} S.M. Cioab\v{a} and M. Tait, \emph{Variations on a theme of Graham and Pollak}, Discrete Mathematics \textbf{313} (2013) 665--676.

\bibitem{graham1} R.L. Graham and H.O. Pollak, \emph{On the addressing problem for loop switching}, Bell System Technical Journal \textbf{50}(8) (1971) 2495--2519.

\bibitem{graham2} R.L. Graham and H.O. Pollak, \emph{On embedding graphs in squashed cubes}, in: Graph Theory and Applications, in: Lecture Notes in Mathematics, vol. 303, Springer, 1972, pp. 99--110.

\bibitem{leader} I. Leader, L. Mili\'{c}evi\'{c} and T.S. Tan, \emph{Decomposing the Complete $r$-Graph}, Journal of Combinatorial Theory, Series A, to appear.

\bibitem{peck} G. Peck, \emph{A new proof of a theorem of Graham and Pollak}, Discrete Mathematics \textbf{49} (1984) 327--328.

\bibitem{tverberg} H. Tverberg, \emph{On the decomposition of $K_n$ into complete bipartite graphs}, Journal of Graph Theory \textbf{6} (1982) 493--494.

\bibitem{vishwanathan1} S. Vishwanathan, \emph{A polynomial space proof of the Graham-Pollak theorem}, Journal of Combinatorial Theory, Series A \textbf{115} (2008) 674--676.

\bibitem{vishwanathan2} S. Vishwanathan, \emph{A counting proof of the Graham-Pollak Theorem}, Discrete Mathematics  \textbf{313}(6) (2013) 765--766.
\end{thebibliography}
\end{document}